\documentclass[12pt]{amsart}
\usepackage[utf8]{inputenc}
\usepackage[english]{babel}

\usepackage[dvipsnames]{xcolor}
\usepackage{setspace}

\usepackage{fullpage}
\usepackage{amsthm}               
\usepackage{amssymb}   
\usepackage{bbm}
\usepackage{tikz}                       
\usetikzlibrary{plotmarks}
\usetikzlibrary{shapes}
\usetikzlibrary{arrows}
\usepackage{graphicx}
\usepackage{caption}
\usepackage{subcaption}
\usepackage{upgreek}
\usepackage{textcomp, gensymb}
\usepackage{capt-of}
\usepackage{ stmaryrd }
\usepackage{colortbl}
\usepackage{multirow}
\usepackage{hhline}
\usepackage[pdftex]{hyperref}
\hypersetup{
    colorlinks=true,
    linkcolor={MidnightBlue},
    citecolor={magenta},
    urlcolor={MidnightBlue}
}
\usepackage{caption}
\usepackage[nameinlink,capitalise]{cleveref}
\usepackage{float}
\usepackage{blkarray}

\usepackage[linesnumbered, algoruled, boxed, lined]{algorithm2e}
\SetKwInOut{Input}{Input}
\SetKwInOut{Output}{Output}
\SetKwIF{If}{ElseIf}{Else}{if}{then}{else if}{else}{}
\SetKwFor{For}{for}{do}{}

\numberwithin{equation}{section}
\theoremstyle{plain}

\newtheorem{theorem}{Theorem}[section]

\newtheorem{corollary}[theorem]{Corollary}
\newtheorem{lemma}[theorem]{Lemma}
\newtheorem{proposition}[theorem]{Proposition}

\theoremstyle{definition}

\newtheorem{remark}[theorem]{Remark}
\newtheorem{example}[theorem]{Example}

\newtheorem{definition}[theorem]{Definition}

\newcommand{\Cc}{{\mathbb C}}
\newcommand{\Nn}{{\mathbb N}}

\newcommand{\Rr}{{\mathbb R}}
\newcommand{\Zz}{{\mathbb Z}}
\newcommand{\Pp}{{\mathbb P}}

\newcommand{\Kk}{\mathbb K}

\newcommand{\cs}{\mathcal{S}}

\newcommand{\bfa}{{\bf a}}

\DeclareMathOperator{\rank}{rank}

\DeclareMathOperator{\init}{in}

\newcommand{\isom}{\cong}

\definecolor{benpurple}{RGB}{180, 0, 240}
\definecolor{joecyan}{RGB}{0,100,100}

\title{Computing Implicitizations of Multi-Graded Polynomial Maps}

\author{Joseph Cummings}
\address{University of Notre Dame }
\email{jcummin7@nd.edu}

\author{Benjamin Hollering}
\address{Technische Universit\"at M\"unchen, 85748 Garching b. München, Boltzmannstr. 3.,  Germany}
\email{benhollering@gmail.com}

\date{}

\begin{document}

\begin{abstract}
    In this paper, we focus on computing the kernel of a map of polynomial rings $\varphi$. This core problem in symbolic computation is known as implicitization. While there are extremely effective Gr\"obner basis methods used to solve this problem, these methods can become infeasible as the number of variables increases. In the case when the map $\varphi$ is multigraded, we consider an alternative approach. We demonstrate how to quickly compute a matrix of maximal rank for which $\varphi$ has a positive multigrading. Then in each graded component we compute the minimal generators of the kernel in that multidegree with linear algebra. We have implemented our techniques in Macaulay2 and show that our implementation can compute many generators of low degree in examples where Gr\"obner techniques have failed. This includes several examples coming from phylogenetics where even a complete list of quadrics and cubics were unknown. When the multigrading refines total degree, our algorithm is \emph{embarassingly parallel} and a fully parallelized version of our algorithm will be forthcoming in OSCAR.
\end{abstract}

\maketitle

\section{Introduction}
Implicitization is a core problem in symbolic computations with many applications in a variety of scientific fields. This problem is focused on computing the kernel of a ring homomorphism $\varphi$
\begin{align*}
    \varphi : R = \Cc[x_1,\dotsc,x_n] &\to S = \Cc[t_1,\dotsc,t_m] \\
                    x_i & \mapsto \varphi_i(x_i)
\end{align*} 
This means one seeks to find a Gr\"obner basis or even just a generating set for the ideal $\ker(\varphi)$. Standard techniques for this typically rely on computing a Gr\"obner basis for the elimination ideal $\langle x_i - \varphi_i(x_i) \rangle$ with respect to an elimination order for the variables $t_j$ \cite{coxlittleoshea}. While modern Gr\"obner bases algorithms are extremely effective at solving a wide array problems, they still often become too expensive as the number of variables and the degree of the polynomials involved grows and are difficult to parallelize effectively \cite{msolve,faugere2002f5, faugere2016complexity}. 

In this paper we focus on computing the kernel of polynomial maps which arise in \emph{algebraic statistics} though our techniques apply more broadly. Many problems in algebraic statistics are fundamentally implicitization problems; however, for many families of interesting statistical models the number of variables involved grows exponentially. For instance, to compute the ideal of phylogenetic invariants for an $n$-leaf phylogenetic tree or network, there are $\approx 4^n$ variables involved \cite{allman2008gmm}. This means that it is often impossible to compute polynomials in $\ker(\varphi)$ for small trees or networks with a computer algebra system. In many algebraic statistics problems one may only need to find a single polynomial in $\ker(\varphi)$ to prove identifiability results \cite{allman2010identifiability, gross2018distinguishing, long2015identifiability} or a collection of statistically meaningful polynomials which can be used for model selection \cite{chifman2014quartet, eriksson2005tree, sturma2023testing}. Even in these cases where only some polynomials in the $\ker(\varphi)$ are needed, modern Gr\"obner bases algorithms which leverage homogeneity and degree-limiting may still fail to compute low-degree polynomials for small examples since there are so many variables involved \cite{martin2023algebraic}. For a thorough (and rather enjoyable) treatment of the Gr\"obner-based approach in the multi-graded setting, we refer the reader to \cite{CCA2}.

In this paper we provide an alternative algorithm to the common Gr\"obner-based approach which exploits the fact that many polynomial maps in algebraic statistics are actually homogeneous in a $\Zz^k$-multigrading. Our approach is inspired by the technique the authors used in \cite{cummings2021invariants} to compute the quadratic polynomials which vanish on certain phylogenetic network models as well as \cite{cummings2023multigraded} where the authors study multigraded Macaulay dual spaces.
In the following section we show how an essentially maximal multigrading in which $\ker(\varphi)$ is homogeneous can be computed without computing $\ker(\varphi)$. We then describe how the generators in $\ker(\varphi)$ which have a given multidegree $\beta \in \Zz^k$ can be computed by solving large linear systems. This means that if $\varphi$ is also homogeneous in the usual sense of total degree, then one compute all generators of total degree $d$ by computing the homogeneous component of $\ker(\varphi)$ with multidegree $\beta \in \Zz^k$ for all $\beta$ which are the multidegree of a monomial of total degree $d$. Moreover, this step is \emph{embarassingly parallel} meaning that the computation of each homogeneous component corresponding to $\beta \in \Zz^k$ can be computed completely in parallel. This makes our algorithm extremely effective at computing all of the low-degree polynomials in the kernel of a polynomial map which is homogeneous in a large multigrading.

The basic idea behind this technique has been noted before in the case that $\varphi$ is homogeneous with respect to the usual $\Zz$-grading given by total degree. However, for many large examples, this technique fails since the linear systems which one needs to solve grow exponentially in the number of variables, which in algebraic statistics often grows exponentially itself. By leveraging multigradings, we are able to instead solve many smaller systems completely in parallel. In our last section we showcase this technique on several examples from algebraic statistics and phylogenetics which Gr\"obner bases techniques or the previously known total-degree version of this algorithm are unable to solve. This includes finding all degree 2 and degree 3 phylogenetic invariants for 4 leaf networks under the Kimura 3-Parameter model. Recently, \cite{martin2023algebraic} attempted this same computation with degree-limited Gr\"obner bases and were unable to find these degree 3 generators even after 100 days of computation time. Another model of interest is the Timura-Nei model \cite{TN93}. This model is more flexible than group-based models and is used more widely in practice. While the vanishing ideal for a generic tree is still currently unkown, the authors in \cite{casanellas2023novel} showed that on an open subset, the ideal for a 4 leaf tree is a complete intersection of dimension 16, and they explicitly produce the ideal. Using our methods, we were able to show that the full ideal is not a complete intersection by exhibiting that there are 375 minimal quadrics in the vanishing ideal.

All of our code along with detailed explanations can be found on our MathRepo page
\begin{center}
    \url{https://mathrepo.mis.mpg.de/MultigradedImplicitization} 
\end{center}
or in the GitHub repository
\begin{center}
        \url{https://github.com/bkholler/MultigradedImplicitization}.
\end{center}
This includes a Macaulay2 \cite{M2} package with our main algorithm implemented. A fully parallelized version of our algorithm will also be available in OSCAR \cite{OSCAR} as soon as this functionality is supported. 

The remainder of this paper is organized as follows. In \cref{sec:MainAlg} we show how multigradings on polynomial maps can be computed and then describe our algorithm which computes $\ker(\varphi)$ up to a given degree. In \cref{sec:PhyloApplications} we examine several different applications of our algorithm to open implicitization problems in phylogenetics which are known to be difficult and show that many low degree polynomials can be found with our algorithm.

\section{The Main Algorithm}
\label{sec:MainAlg}
In this section we show how to compute the homogeneity space of the kernel of a polynomial map $\varphi$ without actually computing $\ker(\varphi)$. The homogeneity space of any ideal in a polynomial ring induces a maximal $\Zz^k$-multigrading in which the ideal is homogeneous. We then leverage this multigrading to give an \emph{embarrassingly parallelizable} algorithm for computing homogeneous components of $\ker(\varphi)$ which is powered by solving large linear systems. Our notation throughout this section is adapted from \cite{jensen2008ComputingGF}.  

\begin{definition}
    Let $f = \sum_{\alpha} c_\alpha x^\alpha \in \Kk[x_1,\dotsc,x_n]$ and let $\omega \in \Rr^n$. Then the \emph{initial form} of $f$ with respect to $\omega$ is 
    \[
        \mathrm{in}_\omega(f) = \sum_{\substack{c_\alpha \neq 0 \\ \omega \cdot \alpha \text{ is minimal}}} c_\alpha x^\alpha.
    \]
    For an ideal $I \subseteq \Kk[x_1,\dotsc,x_n]$ and $\omega \in \Rr^n$, the \emph{initial ideal} of $I$ with respect to $\omega$ is $\mathrm{in}_\omega(I) = \langle \mathrm{in}_\omega(f) ~|~ f \in I\rangle$.
\end{definition}

\begin{definition}
    Let $I \subseteq \Kk[x_1, \ldots, x_n]$ be an ideal. The \emph{homogeneity space} of $I$ is the linear space
    \[
    C_0(I) = \{\omega \in \Rr^n ~|~ \init_{\omega}(I) = I\}.
    \]
\end{definition}

The homogeneity space $C_0(I)$ consists of vectors which yield a grading in which $I$ is homogeneous. Indeed if $\omega \in C_0(I)$ and we set $\deg(x_i) = \omega_i$, then by definition $\mathrm{in}_\omega(f)$ is homogeneous with respect to this grading. Since $\mathrm{in}_\omega(I) = I$, we have that $I$ is generated by homogeneous polynomials, i.e. $I$ has a (possibly non-standard) $\Zz$-grading given by $\omega$. Now, let $b_1, \ldots, b_r \in \Zz^n$ be a basis for $C_0(I)$ and consider the matrix $A = (b_1 | b_2 | \ldots | b_r)^T \in \Zz^{r \times n}$. Then $I$ is homogeneous in the $\Zz^r$-multigrading given by $\deg(x_j) = A_{x_j}$ where $A_{x_j} \in \Zz^r$ is the $j$-th column of $A$ which naturally corresponds to $x_j$. Moreover, this multigrading is maximal in the sense of the following lemma.

\begin{lemma}
\label{lemma:MaximalMultigrading}
Let $A$ be as above and suppose $A' \in \Zz^{k \times n}$ is another matrix for which $I$ is homogeneous in the multigrading induced by $A'$. Then the row space of $A'$ is contained in the row space of $A$. Note that $k$ need not be equal to $r$.
\end{lemma}

\begin{proof}
    Suppose $A'$ is equal to $(b_1' | b_2' | \ldots | b_k')^T \in \Zz^{k \times n}$. It is enough to show that $b_i'$ is in $C_0(I)$ for $i = 1, \dotsc, k$. To this end, consider any homogeneous element $f \in I$ of degree $\beta \in \Zz^k$. Then $f$ is of the form
    \[
        f = \sum_{A' \alpha = \beta} c_\alpha x^\alpha.
    \]
    Then for any $i$, we have that 
    \[
        \mathrm{in}_{b_i'}(f) = \sum_{\substack{A'\alpha = \beta \\ b_i' \cdot \alpha \text{ is minimal}}} c_\alpha x^\alpha
    \]
    As $f$ is homogeneous, $b_i' \cdot \alpha = \beta_i$ for all $\alpha$ appearing in $f$, so it follows that $\mathrm{in}_{b_i'}(f) = f$ for all such homogeneous polynomials. In particular, we conclude that $\mathrm{in}_{b_i'}(I) = I$ and $b_i'\in C_0(I)$ completing the proof.
\end{proof}

We now focus on the problem of finding the homogeneity space of the kernel of a polynomial map. So consider ring homomorphism $\varphi$ of the form
\begin{align*}
    \varphi : R = \Kk[x_1,\dotsc,x_n] &\to S = \Kk[t_1,\dotsc,t_m] \\
                    x_i & \mapsto \varphi(x_i)
\end{align*}
The following theorem and lemma gives an immediate technique to partially compute the homogeneity space of $\ker(\varphi)$.

\begin{lemma}\label{lemma:homogeneitySpace}
    Let $\varphi : R \to S$ be as above, and let $J = \langle x_i - \varphi(x_i) ~|~ i \in [n]\rangle$ be the elimination ideal. Then 
    \[
    C_0(J) = \{ \omega \in \Rr^n ~|~ \mathrm{in}_\omega(x_i - \varphi(x_i)) = x_i - \varphi(x_i) \text{ for all } i \in [n]\}.
    \]
\end{lemma}
\begin{proof}
    One inclusion is obvious. 
    If $\mathrm{in}_\omega(x_i - \varphi(x_i)) = x_i - \varphi(x_i)$ for all $i$, then $\omega \in C_0(J)$. Indeed, if this is the case, then we automatically have that $\mathrm{in}_\omega(J) \supseteq J$. On the other hand, if $f$ is in $J$, then $f = \sum_i h_i (x_i - \varphi(x_i))$ 
    and $\mathrm{in}_\omega(f) = \sum_i \mathrm{in}_\omega (h_i) (x_i - \varphi(x_i))$, so $\mathrm{in}_\omega (f) \in J$. As $\mathrm{in}_\omega(J)$ is generated by all such initial forms we have $\mathrm{in}_\omega(J) = J$.

    Now, let $\omega \in C_0(J)$. Fix an $i \in [n]$ and consider $x_i - \varphi(x_i)$. Since $J$ is homogeneous with respect to $\omega$, we could rewrite $x_i - \varphi(x_i)$ as $\sum_{i=1}^N f_i$ where each $f_i$ satisfies $\mathrm{in}_\omega (f_i) = f_i$ and $f_i \in J$ for all $i\in [N]$. We can assume that $f_1$ has $x_i$ as one of its terms, so $f_1$ takes the form $x_i - \sum_{\alpha} c_\alpha t^\alpha$. The fact that $f_1$ is in the ideal implies that $\sum_{j=2}^N f_j = \varphi(x_i) - (f_1 - x_i)$ lies in $J\cap S$. Since $x_1,\dotsc,x_n$ are algebraically independent, $J\cap S = \{0\}$. Therfore we must have $f_j = 0$ for each $j \geq 2$. It follows that $\mathrm{in}_\omega(x_i - \varphi(x_i)) = x_i - \varphi(x_i)$ for all $i \in [n]$ as claimed.
\end{proof}

\begin{theorem}\label{thm: grading the kernel}
    Let $\varphi: R \to S$ be a homomorphism of polynomial rings as above and $J = \langle x_i - \varphi(x_i) ~|~ i \in [n]\rangle \subseteq \Kk[x_1, \ldots, x_n, t_1, \ldots t_m]$ be the associated elimination ideal. Let $b_1, \ldots,  b_r \in \Zz^{n+m}$ be a basis for the homogeneity space $C_0(J)$ and write $b_i = (b_i', b_i'') \in \Zz^{n+m}$. Then $b_1', \ldots, b_r' \in \Zz^n$ are contained in the homogeneity space of $\ker(\varphi)$.
\end{theorem}
\begin{proof}
    Let $i \in [n]$. We will show that $b_i' \in C_0(\ker(\varphi))$. By Lemma \ref{lemma:homogeneitySpace} the generators of $J$, i.e. $x_j - \varphi(x_j)$, are all homogeneous with respect to the $\Zz$-grading given by $b_i$. Now fix a lexicographic term ordering where $t_i \succ x_j$ for all $(i,j) \in [m] \times [n]$. In order to compute $\ker(\varphi)$, we need to compute a Gr{\"o}bner basis with respect to $\prec$. Our key insight is that each step in Buchberger's algorithm always adds homogeneous (with respect to $b_i$) polynomials to the generating set. This is because an $S$-pair of two homogenous polynomials is also homogeneous and the reduction of an $S$-pair by homogeneous polynomials will also be homogeneous. 

    Now let $\mathcal{G}$ be the resulting Gr{\"o}bner basis and let $\mathcal{G}' = \mathcal{G} \cap R$. As discussed above, each element of $\mathcal{G}'$ is homogeneous with respect to $b_i$, but as they involve no $t_j$'s, they are also homogeneous with respect to $b_i'$. It follows that $\mathrm{in}_{b_i'}(\mathcal{G}') = \mathcal{G}'$ and that $\mathrm{in}_{b_i'}(\ker(\varphi)) = \ker(\varphi)$.
\end{proof}

\begin{remark}
    It is important to note that the elements in the homogeneity space of $\ker(\varphi)$ obtained from \cref{thm: grading the kernel} are generally not independent nor spanning. We will see in \cref{exa: grassmannian multigrading}, that the rows of $A$ are not independent. For an example where they aren't spanning, consider the following map.
    \begin{align*}
    \varphi : \Kk[x,y,z] &\to \Kk[a,b]\\
        x &\mapsto (a+b)^2 \\
        y &\mapsto a^2 - b^2 \\
        z &\mapsto (a-b)^2
    \end{align*}
    The kernel of this map is the toric ideal $\langle xz - y^2\rangle$; however, using \cref{thm: grading the kernel}, we only detect a 1-dimensional subspace of the 2-dimensional homogeneity space. 

\end{remark}

\begin{corollary}
\label{cor:MaxGrading}
    Let $\varphi: R \to S$ be a homomorphism of polynomial rings, $J$ be the associated elimination ideal, and $b_i = (b_i', b_i'') \in \Zz^{n+m}, ~ i \in [k]$ be a basis for $C_0(J)$. 
    Let $A = (b_1 | b_2 | \ldots | b_r)^T \in \Zz^{r \times (n+m)}$. Then $\varphi$ is homogeneous in the multigrading given by $\deg(t_j) = A_{t_j}$ and $\deg(x_j) = A_{x_j}$. Moreover, $\ker(\varphi)$ is homogeneous in the induced multigrading which is $\deg(x_j) = A_{x_j}$. 
\end{corollary}
    
Note that in the previous corollary we make the natural identification between the columns of the matrix $A$ with the corresponding variables in the polynomial rings $R$ and $S$ for simplicity of notation. The following example illustrates this corollary. 

\begin{example}\label{exa: grassmannian multigrading}
    Consider the Pl{\"u}cker embedding of $\mathrm{Gr}(2,4)$. Set $R = \Cc[p_{ij} ~|~ 1 \leq i < j \leq 4]$ and $S = \Cc[x_{ij} ~|~ i \in [2], j \in [4]] $.
    \begin{align*}
        \varphi :  R&\to S \\
        p_{ij} &\mapsto \det (M_{ij})
    \end{align*}
    Here $M$ is a $2 \times 5$ matrix whose entries are the variables $x_{ij}$ and $M_{ij}$ corresponds to the square $2\times 2$ sub-matrix of $M$ whose columns are the $i^\text{th}$ and $j^\text{th}$ columns of $M$. If we let $J$ be the elimination ideal of $\varphi$, the homogeneity space $C_0(J) \subset \Rr^5$ is the rowspan of the following matrix.

    \[
    \begin{blockarray}{cccccccccccccc}
    p_{1,2} & p_{1,3} & p_{2,3} & p_{1,4} & p_{2,4} & p_{3,4} & x_{1,1} & x_{1,2} & x_{1,3} & x_{1,4} & x_{2,1} 
    & x_{2,2} & x_{2,3} & x_{2,4} \\
    \begin{block}{(cccccc|cccccccc)}
      1&1&1&1&1&1&1&1&1&1&0&0&0&0\\
      1&1&0&1&0&0&1&0&0&0&1&0&0&0\\
      1&0&1&0&1&0&0&1&0&0&0&1&0&0\\
      0&1&1&0&0&1&0&0&1&0&0&0&1&0\\
      -1&-1&-1&0&0&0&-1&-1&-1&0&0&0&0&1\\
      \end{block}
      \end{blockarray}
    \]

    The lattice spanned by the first 6 columns of the matrix above has rank 4, so we deduce that there is an action of $(\Cc^\times)^4$ on the affine cone of $\mathrm{Gr}(2,4)$. The row of ones corresponds to the usual scaling action on $\Cc^6$ used to construct $\Pp^5 \isom \Cc^6 / \Cc^\times$, so there is a 3 dimensional torus $(\Cc^\times)^4/\Cc^\times$ acting on $\mathrm{Gr}(2,4) \subseteq \Pp^5$.
    
\end{example}

So given any polynomial map $\varphi$, \cref{cor:MaxGrading} allows one to inexpensively compute a multigrading in which $\ker(\varphi)$ is homogeneous. We now focus on the task of computing minimal generators of $\ker(\varphi)$ of a fixed total degree $d$. One naive way of doing this is to consider an arbitrary element 
\begin{equation}
\label{eqn:GenericTotalDegPoly}
    f^{(d)} = \sum_{\substack{ \alpha \\ \alpha \cdot \mathbbm{1} = d}} c_\alpha x^\alpha
\end{equation}
of total degree $d$. Then of course we know that $f \in \ker(\varphi)$ if and only if $\varphi(f) = 0$. Observe that we can simply compute $\varphi(f)$, collect the coefficients of each monomial $t^\gamma$, and set each of these to 0. This gives us necessary and sufficient linear conditions on the coefficients $c_\alpha$. Computing a basis for the set of all such $c_\alpha$ then gives us a set of minimal generators of degree $d$ for $\ker(\varphi)$. This is demonstrated by the following example. 

\begin{example}\label{exa: grassmannian quadrics}
    We continue with \cref{exa: grassmannian multigrading} by finding the quadrics in $I = \ker(\varphi)$. Consider a generic quadric in $R$.
    \begin{align*}
    f^{(2)} &= c_{\left(1,\,2\right),\left(1,\,2\right)}p_{1,2}^{2}+c_{\left(1,\,2\right),\left(1,\,3\right)}p_{1,2}p_{1,3}+c_{\left(1,\,3\right),\left(1,\,3\right)}p_{1,3}^{2}+c_{\left(1,\,2\right),\left(2,\,3\right)}p_{1,2}p_{2,3}+c_{\left(1,\,3\right),\left(2,\,3\right)}p_{1,3}p_{2,3}+\\
    &\;\;\;\;\;c_{\left(2,\,3\right),\left(2,\,3\right)}p_{2,3}^{2}+c_{\left(1,\,2\right),\left(1,\,4\right)}p_{1,2}p_{1,4}+c_{\left(1,\,3\right),\left(1,\,4\right)}p_{1,3}p_{1,4}+c_{\left(2,\,3\right),\left(1,\,4\right)}p_{2,3}p_{1,4}+c_{\left(1,\,4\right),\left(1,\,4\right)}p_{1,4}^{2}+\\
    &\;\;\;\;\;c_{\left(1,\,2\right),\left(2,\,4\right)}p_{1,2}p_{2,4}+c_{\left(1,\,3\right),\left(2,\,4\right)}p_{1,3}p_{2,4}+c_{\left(2,\,3\right),\left(2,\,4\right)}p_{2,3}p_{2,4}+c_{\left(1,\,4\right),\left(2,\,4\right)}p_{1,4}p_{2,4}+\\
    &\;\;\;\;\;c_{\left(2,\,4\right),\left(2,\,4\right)}p_{2,4}^{2}+c_{\left(1,\,2\right),\left(3,\,4\right)}p_{1,2}p_{3,4}+c_{\left(1,\,3\right),\left(3,\,4\right)}p_{1,3}p_{3,4}+c_{\left(2,\,3\right),\left(3,\,4\right)}p_{2,3}p_{3,4}+ \\
    &\;\;\;\;\;c_{\left(1,\,4\right),\left(3,\,4\right)}p_{1,4}p_{3,4}+c_{\left(2,\,4\right),\left(3,\,4\right)}p_{2,4}p_{3,4}+c_{\left(3,\,4\right),\left(3,\,4\right)}p_{3,4}^{2}
    \end{align*}
    As stated above, we can apply $\varphi$ to $f^{(2)}$, collect coefficients, and get necessary and sufficient linear conditions on the $c_{(i,j),(k,\ell)}$'s to find a basis for the the kernel of $\varphi$ in degree 2. This is more than a little cumbersome, so we will forego showing you this computation explicitly; however, we will describe how to implement this in your favorite computer algebra system.
    
    There are $\binom{6 + 2 - 1}{2} = 21$ monomials spanning $R_2$ and the images of each monomial are supported on 72 monomials of degree 4 in $S$. In order to find a basis for $\ker(\varphi)$ in degree 2, we need to find the \emph{linear} relations among the polynomials $\varphi(p_{ij}p_{k\ell})$. This amounts to finding the kernel of a $72 \times 21$ matrix $C$. The columns of this matrix are indexed by the monomials spanning $R_2$ and the rows are indexed by the monomials in $S_4$ on which $\varphi(p_{ij}p_{k\ell})$ are supported. The entry $C_{x^\alpha, p_{ij}p_{k\ell}}$ is the coefficient of $x^{\alpha}$ in $\varphi(p_{ij}p_{k\ell})$. The kernel of $C$ is generated by exactly one element and it corresponds to the Pl{\"u}cker relation $p_{2,3}p_{1,4}-p_{1,3}p_{2,4}+p_{1,2}p_{3,4}$.
\end{example}

While the previous approach can be used occasionally it is often not helpful since the generic polynomial $f$ which one has to consider has $\binom{n + d- 1}{d}$ terms which grows exponentially in $d$. However, if we instead apply this technique to each homogeneous component of a finer multigrading, we can solve much smaller linear systems instead. So let $A \in \Zz^{r \times n}$ be a maximal rank multigrading on $\ker(\varphi)$ and assume that the $\mathbbm{1} \in \mathrm{rowspan}(A)$ which guarantees that $\ker(\varphi)$ is also homogeneous in the typical sense of total degree. This implies that $\ker(\varphi)$ has a minimal generating set $\{f_1, \ldots f_k \}$ consisting of polynomials where each minimal generator $f_i$ is homogeneous in the multigrading determined by $A$. So if we wish to compute a set of minimal generators of $\ker(\varphi)$ which are total degree $d$, then we can instead consider each multidegree $A \alpha = \beta \in \Nn A$ such that $\alpha \cdot \mathbbm {1} = d$ separately. This means to compute all degree $d$ minimal generators,  we no longer consider a polynomial of the form found in \cref{eqn:GenericTotalDegPoly} but instead we consider a polynomial of the form
\begin{equation}
\label{eqn:GenericMultidegPoly}
    f^{(\beta)} = \sum_{\substack{\alpha \\ A \alpha = \beta}} c_\alpha x^\alpha
\end{equation}
for each homogeneous component of degree $\beta \in \Nn A$ such that $\beta = A \alpha$ and $\alpha \cdot \mathbbm{1} = d$.  Generally, as the multigrading $A$ becomes finer, meaning $\rank(A)$ becomes larger, the size of the monomial basis $M_\beta = \{\alpha \in \Nn^n ~|~ A \alpha = \beta \}$ for the homogeneous component $R_\beta = \{f \in R ~|~ \deg(f) = \beta \}$ becomes smaller. This means instead of solving one extremely large linear system which corresponds to $\varphi(f^{(d)}) = 0$ as we would get from \cref{eqn:GenericTotalDegPoly}, we can solve many small linear systems which come from settings $\varphi(f^{(\beta)}) = 0$ for each $\beta$. The following example elucidates this. 

\begin{example}\label{exa: grassmannian quadrics with multigrading}
    We continue with our running example of $\mathrm{Gr}(2,4)$. As we saw in \cref{exa: grassmannian quadrics}. There was a single quadratic Pl{\"u}cker relation. However, we had to compute the kernel of a $72 \times 21$ matrix. Using the multigrading from \cref{exa: grassmannian multigrading}, we can greatly reduce the size of this computation. The total degree 2 component of $R$ can be divided into 19 homogeneous components using the grading matrix from \cref{exa: grassmannian multigrading}. Only one of these homogeneous components has a basis with more than a single element. This is $R_{(2,1,1,1,-1)}$ and is spanned by $M_{(2,1,1,1,-1)} = \{p_{1,2}p_{3,4},p_{1,3}p_{2,4}, p_{2,3}p_{1,4}\}$. There will be no relations supported on the other components since there are evidently no monomials in $\ker(\varphi)$. Consider a generic polynomial of degree $(2,1,1,1,-1)$.
    \[
        f^{(2,1,1,1,-1)} = c_1 p_{1,2}p_{3,4}+ c_2p_{1,3}p_{2,4}+c_3 p_{2,3}p_{1,4}.
    \]
    Now, we can find the necessary and sufficient linear relations among the $c_i$'s to ensure $f^{(2,1,1,1,-1)}$ is in $\ker(\varphi)$. This can be done as follows.

    There are 6 monomials that appear when we apply $\varphi$ to the monomial basis of $R_{(2,1,1,1,-1)}$, so we construct a $6\times 3$ matrix whose columns are indexed by the elements of $M_{(2,1,1,1,-1)}$ and whose rows are indexed by these 6 monomials. The entry in row $x^\alpha$ and column $p_{i,j}p_{k,\ell}$ is the coefficient of $x^\alpha$ in $\varphi(p_{i,j}p_{k,\ell})$.
    \[\begin{blockarray}{cccc}
     & p_{1,2}p_{3,4} &p_{1,3}p_{2,4}&p_{2,3}p_{1,4} \\
    \begin{block}{l(ccc)}
    x_{1,3}x_{1,4}x_{2,1}x_{2,2} & 0&1&1\\
    x_{1,2}x_{1,4}x_{2,1}x_{2,3} & 1&0&-1\\
    x_{1,1}x_{1,4}x_{2,2}x_{2,3} & -1&-1&0\\
    x_{1,2}x_{1,3}x_{2,1}x_{2,4} & -1&-1&0\\
    x_{1,1}x_{1,3}x_{2,2}x_{2,4} & 1&0&-1\\
    x_{1,1}x_{1,2}x_{2,3}x_{2,4} & 0&1&1 \\
      \end{block}
      \end{blockarray}
    \]
    The kernel of this matrix is spanned by $(1,-1,1)^T$ giving us the Pl{\"u}cker relation $p_{1,2}p_{3,4} - p_{1,3}p_{2,4}+p_{2,3}p_{1,4}$. 
\end{example}

This idea gives an immediate algorithm for computing all of the minimal generators of $\ker(\varphi)$ of degree at most $d$ which can be found at the end of this section. First, we note that while building the set of minimal generators in degree $d$, one may further reduce the set $M_\beta$ for each $\beta$ such that $\beta \cdot \mathbbm{1} = d$ using the set of generators of degree strictly less than $d$. This idea is captured by the following proposition and example. 

\begin{proposition}
\label{prop:TrimMonomialBasis}
Suppose $H$ is a minimal homogeneous generating set for $\ker(\varphi)$. Let $d$ be a positive integer, let $G = \{g \in H ~|~ \deg(g) = \beta_g = A\alpha_g \text{ with } \alpha_g \cdot \mathbbm{1} < d\}$ and let $\beta = A \alpha$ where $\alpha \cdot \mathbbm{1} = d$. Consider the vector space 
\[
\mathrm{Lift}(G) = \mathrm{span}_{\mathbb{K}} \{x^{\gamma} g ~|~ \deg(x^\gamma) = \beta - \deg(g) \text{ and } g \in G\} \subseteq R_\beta.
\] We can write $R_\beta$ as a direct sum $\mathrm{Lift}(G) \oplus V_\beta$. Then the minimal generators of $\ker(\varphi)$ of degree $\beta$ can be chosen to be supported on $V_\beta$.
\end{proposition}

\begin{proof}
    Suppose $f$ is a minimal generator of degree $\beta$. This polynomial can be rewritten as $f_G + f_{V_\beta}$ where $f_G \in \mathrm{Lift}(G)$ and $f_{V_\beta} \in V_\beta$. By definition, $f_G$ is in the ideal generated by $G$; hence, $f_G \in \ker(\varphi)$. It follows that $f_{V_\beta} \in \ker(\varphi)$, and this polynomial can be chosen as a minimal generator instead of $f$.
\end{proof}

\begin{example}
We illustrate \cref{prop:TrimMonomialBasis} by continuing our running example, $\mathrm{Gr}(2,4)$. We will compute minimal generators of $\ker(\varphi)$ of degree $(3,1,1,2,-1) = \deg(p_{3,4}) + \deg(f)$ where $f$ is the quadratic Pl{\"u}cker relation in $\ker(\varphi)$ from \cref{exa: grassmannian quadrics} and \cref{exa: grassmannian quadrics with multigrading}. It is well known that the Pl{\"u}cker relation forms a universal Gr{\"o}bner basis for this ideal; therefore, we should find that there are no minimal generators of degree $(3,1,1,2,-1)$. 

If we continue as before, we would compute the kernel of the matrix below.
\[
\begin{blockarray}{ccccccccccc}
\begin{block}{c(cccccccccc)}
p_{1,2}p_{3,4}^{2} & 0 & -1 & 1 & 0 & 2 & -2 & 0 & -1 & 1 & 0 \\
p_{1,3}p_{2,4}p_{3,4} & -1 & 0 & 1 & 1 & 1 & -1 & -1 & -1 & 0 & 1 \\
p_{2,3}p_{1,4}p_{3,4} & -1 & 1 & 0 & 1 & -1 & 1 & -1 & 0 & -1 & 1 \\
\end{block}
\end{blockarray}^T
\]
The rows are indexed by the 10 degree 6 monomials in $S$ which appear after applying $\varphi$ to the monomial basis of $R_{(3,1,1,2,-1)}$. The kernel is generated by $(1,-1,1)^T$, and it corresponds to $p_{3,4} f$. This is clearly not a minimal generator.

Instead of considering the monomial basis, we could have used the basis 
\[\{p_{3,4} f\}\cup \{p_{1,3}p_{2,4} p_{3,4}, p_{2,3}p_{1,4}p_{3,4}\}\] 
of $R_{(3,1,1,2,-1)}.$ Note that $\{p_{3,4}f\}$ is a basis for $\mathrm{Lift}(f)$ as in \cref{prop:TrimMonomialBasis} and the second set of monomials is a basis for $V_{(3,1,1,2,-1)}$. Since it is already evident that $p_{3,4} f \in \ker(\varphi)$, we only need to search for linear relations among $\varphi(p_{1,3}p_{2,4} p_{3,4})$ and $\varphi(p_{2,3}p_{1,4}p_{3,4})$. Of course, there are none since any such linear relation corresponds to an element of the kernel of the matrix above with the first column removed. Since this kernel is trivial, we see that there are no minimal generators of degree $(3,1,1,2,-1)$.
\end{example}

Lastly, we discuss an additional speed-up based on \cite{hollering2021identifiability, rosen2014computing} which uses the following proposition to throw out some homogeneous components that cannot have generators.

\begin{proposition}\cite{rosen2014computing}
\label{prop:AlgebraicMatroid}
    Let $\varphi : R = \Kk[x_1,\dotsc,x_n] \to S = \Kk[t_1,\dotsc,t_m] $ be a ring homomorphism, $J(\varphi)$ be the matrix 
    $J(\varphi)_{ij} = \left(\frac{\partial \varphi_j}{\partial t_i} \right)$, and $S \subset [n]$. Then
    \[
    \Kk[x_i ~|~ i \in S] \cap \ker(\varphi) = \langle 0 \rangle \iff \rank(J(\varphi)_S) = |S|
    \]
\end{proposition}

\begin{remark}
For those familiar with matroid theory, the previous proposition essentially states that the \emph{algebraic matroid} defined by the prime ideal $\ker(\varphi)$ is the same as the linear matroid defined by $J(\varphi)$ over the fraction field $\Kk(t_1, \ldots t_m)$. For a more detailed discussion of different cryptomorphic constructions of algebraic matroids we refer the reader to \cite{hollering2021identifiability, rosen2014computing}. 
\end{remark}

Now suppose we want to compute the degree $\beta$ homogeneous component of $\ker(\varphi)$. Let $S \subseteq \{x_1, \ldots, x_n \}$ be the subset of the variables which $M_\beta$ is supported on. Then by \cref{prop:AlgebraicMatroid},
if $\rank(J(\phi)_S) = |S|$, then there are no generators in $\ker(\phi)$ whose support is $S$. This immediately implies the following corollary. 

\begin{corollary}
\label{cor:MatroidComponentSkip}
Let $\varphi : R = \Kk[x_1,\dotsc,x_n] \to S = \Kk[t_1,\dotsc,t_m] $ be a ring homomorphism, $J(\varphi)$ be the matrix 
$J(\varphi)_{ij} = \left(\frac{\partial \varphi_j}{\partial t_i} \right)$. Let $M_\beta$ be a monomial basis for the homogeneous component of degree $\beta$ of $\ker(\varphi)$ and $S \subset [n]$ correspond to the subset of variables on which $M_\beta$ is supported. If $\rank(J(\varphi)_S) = |S|$ then there are no generators of degree $\beta$ in $\ker(\varphi)$
\end{corollary}

What makes the above corollary extremely powerful for the purpose of the task at hand is the observation from \cite{rosen2014computing}, that if one plugs in random values for the variables $t_j$ into $J(\varphi)$, then \cref{prop:AlgebraicMatroid} still holds with probability 1. This means that for the purposes of our algorithm, we can simply compute the matrix $J(\varphi)$ and then substitute in random values for the parameters $t_j$. This allows us to skip over many components which can never yield generators by simply computing the rank of $J(\varphi)_S$ which is a $m \times |S|$ matrix with entries in $\Kk$ which is extremely cheap compared to the time it takes to evaluate $\varphi$ on $M_\beta$. Also, we note that when plugging in random values for the $t_j$, the rank of $J(\varphi)_S$ can only drop. This means in our algorithm we would just unnecessarily compute the component of $\ker(\varphi)$ of degree $\beta$. Thus even though we are leveraging some numerical speed-ups, the output is always still correct. In the next section we will show how effective this step can be at reducing the total computation time on several large examples.

We end the section with Algorithm \ref{alg:componentsOfKernel} which naturally arises from the above discussion. Observe that one major advantage that this algorithm has over other approaches comes from parallelization. The inner loop in Algorithm \ref{alg:componentsOfKernel} below which runs over all multidegrees $\beta$ that correspond to degree $d$ monomials is \emph{embarassingly parallel}. This means that massive speedups can be achieved if the algorithm is run in parallel on a large cluster. In our last section we showcase how effective this algorithm can be on some difficult examples from phylogenetics. 

\begin{remark}
    Throughout the latter part of this section, we assumed that $\mathbbm{1}$ was in the row-space of $A$. If instead we just assumed that the grading were positive i.e. there is a vector $\bfa = (a_1,\dotsc,a_n) \in \Zz_{>0}^n$ in the row-space of $A$, you can still construct an algorithm similar to Algorithm 1. The only difference would be to replace $P_i = \{\beta \in \Nn A ~|~ \beta = A \alpha, ~ \alpha \cdot \mathbbm{1} = i\}$ with the set $P_i' = \{\beta \in \Nn A ~|~ \beta = A \alpha, ~ \alpha \cdot \bfa = i\}$.
\end{remark}

\begin{algorithm}
\Input{A polynomial map $\varphi: R = \Kk[x_1, \ldots, x_n] \to S = \Kk[t_1, \ldots t_m]$ and a total degree $d$}
\Output{All minimal generators of $\ker(\varphi)$ of total degree at most $d$}

Compute the homogeneity space of $\langle x_i - \varphi(x_i) \rangle$ and set $A \in \Zz^{r \times n}$ to be the associated multigrading on $\ker(\varphi)$\\
Set $G := \{\}$ to be the set of computed minimal generators\\
Set $B := \{\}$ to be the list of monomial bases for each multidegree\\
Set $J := J(\varphi)$ and substitute random values for the $t_j$

\For{$i = 1$ \KwTo $d$}{

    Set $P_i := \{\beta \in \Nn A ~|~ \beta = A \alpha, ~ \alpha \cdot \mathbbm{1} = i\}$
    
    \For{$\beta \in P_i$}{

        Compute the monomial basis $M_\beta$ for $R_\beta$ and set $S$ to be the variable support of $M_B$
        
        \If{$\rank(J(\varphi)_S) = |S|$}{
            \textbf{continue}
        }
    
         Compute a basis for $V_\beta$ where $R_\beta = \mathrm{Lift}(G) \oplus V_\beta$ using Proposition \ref{prop:TrimMonomialBasis} 

        Expand $\displaystyle \varphi(f^{(\beta)})$ and extract the linear system $L_\beta c = 0$ where $c = (c_\alpha)$ which is obtained by setting $\varphi(f^{(\beta)}) = 0$.

        Compute a basis $v^{(1)}, v^{(2)}, \ldots, v^{(\ell)} \in \Kk^{M_\beta}$ for $\ker(L_\beta)$

        $G = G \cup \{\sum_{\alpha \in M_\beta} v_\alpha^{(i)} x^\alpha ~|~ i \in [\ell] \}$
    }
}
\Return{$G$}
\caption{\texttt{componentsOfKernel}}
\label{alg:componentsOfKernel}
\end{algorithm}

\section{Applications to Algebraic Statistics and Phylogenetics}
\label{sec:PhyloApplications}
In this section we apply Algorithm \ref{alg:componentsOfKernel} to find low-degree minimal generators for several examples in algebraic statistics which come from mathematical phylogenetics. These examples have been previously shown to be extremely difficult and Gr\"obner basis algorithms typically do not terminate when applied to them even when degree-limiting is utilized \cite{martin2023algebraic}. In all of the cases which we describe below, we have used the Macaulay2 implementation of our algorithm which is not parallelized since this functionality. This means that the main advantage of this technique is not being fully leveraged in the below examples. Despite this, the algorithm still performs extremely well. 
All of the code for constructing the polynomial maps below can be found at our MathRepo page \cite{mathrepo}.

\subsection{The General Markov Model on a Phylogenetic Tree}
In this subsection we provide a very brief overview of phylogenetic Markov models and the general Markov model. Since our main purpose here is to simply showcase the effectiveness of this algorithm on some notoriously difficult polynomial maps, we do not provide significant detail or background on phylogenetics and describe the polynomial maps involved primarily from an algebraic perspective. For a more detailed discussion on phylogenetics we refer the reader to \cite{steel2016phylogeny, algstat}.

A $\kappa$-state phylogenetic Markov model on a $n$-leaf, leaf-labelled rooted binary tree $T$ is a \emph{directed acyclic graphical model} in which all of the internal nodes are hidden. The model produces a joint distribution on all possible joint states $(i_1, \ldots, i_n) \in [\kappa]^n$ which can be observed at the leaves of $T$. This distribution is determined by associating a $\kappa$-state random variable $X_v$ to each internal vertex $v$ of $T$ and a $\kappa \times \kappa$ transition matrix $M^e$ to each directed edge $e = (u,v)$ of $T$ such that $M_{i,j}^e = P(X_v = j | X_u = i)$. A root distribution $\pi$ for the root $\rho$ of $T$ is also needed. Then the probability of observing
$(i_1, \ldots i_n) \in [\kappa]^n$ of states at the leaves is 
\[
p_{i_1 \ldots i_n} = P(X_1 = i_1, \ldots, X_n = i_n) ~= 
\sum_{j \in [\kappa]^{Int(T)}}\pi_{j_\rho}\prod_{(u,v) \in E(T)}M_{j_u, j_v}^{(u,v)}. 
\]
which as we can see is a polynomial expression in the parameters $M_{i,j}^e$ and $\pi_k$. This means the model can essentially be viewed as the image of a polynomial map, and thus the vanishing ideal of the model is the kernel of the map below.
\begin{align}
\label{eqn:GMM_TreeParam}
    \psi_T : \Kk[p_{i_1 \ldots i_n} ~|~ (i_1, \ldots, i_n) \in [\kappa]^n] &\to \Kk[M^e_{i,j}, \pi_k ~|~ e \in E(T), i,j,k \in [\kappa]] \\
    p_{i_1 \ldots i_n} &\mapsto \sum_{j \in [\kappa]^{Int(T)}}\pi_{j_\rho}\prod_{(u,v) \in E(T)}M_{j_u, j_v}^{(u,v)} \nonumber
\end{align}
If no other restrictions are made on the transition matrices $M^e$ and the root distribution $\pi$, then resulting phylogenetic model is called the \emph{general Markov model} \cite{allman2008gmm}. For any algebraic phylogenetic model, $\psi_T$, the kernel $\psi_T$, denoted $I_T$, is often called the \emph{ideal of phylogenetic invariants} of the model. The number of variables involved here grows exponentially in the number of leaves $n$ of the tree $T$. This means for large trees it is often impossible to compute the kernel of $\psi_T$ with standard methods. 

Finding a complete set of generators for $\ker(\psi_T)$ when $n = 3$ and $\kappa = 4$ is still an open question, though the \emph{Salmon Conjecture} \cite{allman2003invariants, allman2008gmm} contains a conjectural set of generators which have been shown to define the model set theoretically \cite{friedlan2012salmon}. Further numerical evidence has also been found in \cite{bates2011salmon}.

We tried to find all degree 5 polynomials in the kernel which are known with our \texttt{Macaulay2} implementation of Algorithm 1. In this case there are $\binom{64+5-1}{5} = 10424128$ monomials which yield a total of 175616 unique multidegrees. While our current Macaulay2 implementation was able to compute some components, our current estimate is that it would take approximately 130 hours to compute all components, but it typically runs out of RAM. Based on our current benchmarks, we expect these issues to be solved by our OSCAR implementation. We end this section with a short application of our algorithm to the easier problem of when $\kappa = 3$. 

\begin{example}\label{exa:salmon}
    When $\kappa = 3$, it is known that the $\ker(\psi_T)$ is cut out by 27 quartics \cite{Pachter-Sturmfels04}. We were able to verify that there are indeed 27 minimal quartics using our unparallelized Macaulay2 implementation in 29.76 seconds. We also tried to verify this using Gr\"obner bases; however, we killed this computation after an 76 minutes. 
\end{example}

\subsection{The K3P Model on a Phylogenetic Network}
Another well studied family of phylogenetic models are \emph{group-based models}. These models have been studied extensively from an algebraic perspective \cite{draisma2009equivariant, evans1993invariants, hendy1996complete, mateusz2011geometry, sturmfels2005toric} and many algebraic problems are well understood including a complete description of the Gr\"obner basis for the vanishing ideal of the model \cite{sturmfels2005toric}. This is because these models allow for a linear change of coordinates \cite{evans1993invariants, hendy1996complete} in which the parameterization of the model becomes a monomial map and thus the vanishing ideal becomes toric \cite{sturmfels2005toric}. While group-based models on trees are relatively well understood, more interest recently in phylogenetics has been focused on \emph{phylogenetic networks} which will be our main focus in this subsection. We begin with a description of the monomial parameterization for trees since this will be used to define the network parameterization. 

In a group-based model, the states of the random variables involved are identified with the elements of a finite abelian group $G$. This allows a simultaneous coordinate change on both the domain and codomain of $\psi_T$ which essentially comes from applying to the discrete Fourier transform to the expression for the joint probabilities Equation \ref{eqn:GMM_TreeParam}. For a more detailed explanation of this coordinate change we refer the reader to \cite[Chapter 15]{algstat} and instead focus on defining the polynomial map in this new coordinate system which is what we will run our algorithm on. 

The transformed coordinates of the domain of $\psi_T$ are denoted by $q_{g_1 \ldots g_n}$ and are typically called the \emph{Fourier coordinates}. We then have new parameters $a_g^e$ for each edge $e \in E(T)$ and $g \in G$. Since $T$ is a tree, removing any edge $e$ of $T$ naturally induces a partition of the leaf set into two connected components which is called a \emph{split} of $T$ and is denoted by $A_e | B_e$. The parameterization of the model in these coordinates is given by
\begin{equation}
\label{eqn:FourierTreeParam}
q_{g_1,\ldots g_n} = 
\begin{cases}
\prod_{e \in E(T)} a_{\sum_{i \in A_e} g_i}^e & \mbox{ if } \sum_{i \in [n]}g_i = 0 \\ 
0 & \mbox{ otherwise}
\end{cases}
\end{equation}

Many well known phylogenetic models are group-based such as the Cavendar-Farris-Neyman model, the Jukes-Cantor model, the Kimura 2-Parameter model, and the Kimura 3-Parameter (K3P) model which is typically the most difficult to compute and will be our main object of interest later in this subsection. As discussed previously, group-based models on trees are relatively well understood but many open questions remain. The simplest type of network from an algebraic perspective is called a \emph{sunlet network} and was first introduced in \cite{gross2018distinguishing} and further studied algebraically in \cite{cummings2021invariants}.

\begin{definition}
    A $n$-\emph{sunlet network} is a semi-directed graph with a distinguished vertex called the \emph{reticulation vertex} and whose underlying graph is obtained by adding a leaf to every vertex of a $n$-cycle and then directing the non-leaf edges which are adjacent to the reticulation vertex towards it. 
\end{definition}

\begin{figure}
    \centering
    \begin{subfigure}[b]{0.3\linewidth}
        \centering
        \begin{tikzpicture}[scale = .5, thick]
        \draw [dashed] (2,2)--(4,2);
        \draw (4,2)--(4,4);
        \draw (4,4)--(2,4);
        \draw [dashed] (2,4)--(2,2);
        
        \draw (2,2)--(1,1);
        \draw (4,2)--(5,1);
        \draw (4,4)--(5,5);
        \draw (2,4)--(1,5);
        
        \draw (1,1) node[below]{$1$};
        \draw (5,1) node[below]{$4$};
        \draw (5,5) node[above]{$3$};
        \draw (1,5) node[above]{$2$};
        
        \draw (1,1) node[right]{$e_1$};
        \draw (5,1) node[left]{$e_4$};
        \draw (5,5) node[left]{$e_3$};
        \draw (1,5) node[right]{$e_2$};
        
        \draw (3,2) node[below]{$e_8$};
        \draw (4,3) node[right]{$e_7$};
        \draw (3,4) node[above]{$e_6$};
        \draw (2,3) node[left]{$e_5$};
        \end{tikzpicture}
        \caption{$\cs_4$}
    \end{subfigure}
    \begin{subfigure}[b]{0.3\linewidth}
        \begin{tikzpicture}[scale = .5, thick]
        \draw (0,0)--(2,2);
        \draw (2,2)--(4,2);
        \draw (4,2)--(6,0);
        \draw (2,2)--(1,3);
        \draw (4,2)--(5,3);
        
        \draw [fill] (0,0) circle [radius = .1];
        \draw [fill] (1,1) circle [radius = .1];
        \draw [fill] (2,2) circle [radius = .1];
        \draw [fill] (1,3) circle [radius = .1];
        \draw [fill] (4,2) circle [radius = .1];
        \draw [fill] (5,3) circle [radius = .1];
        \draw [fill] (5,1) circle [radius = .1];
        \draw [fill] (6,0) circle [radius = .1];
        
        \draw (0,0) node[below left] {$1$};
        \draw (1,3) node[above left] {$2$};
        \draw (5,3) node[above right] {$3$};
        \draw (6,0) node[below right] {$4$};
        
        \draw (.25,.25) node[right] {$e_1$};
        \draw (1.25,2.75) node[right] {$e_2$};
        \draw (4.75,2.75) node[left] {$e_3$};
        \draw (5.75,.25) node[left] {$e_4$};
        \draw (1.25,1.25) node[right] {$e_5$};
        \draw (3,2) node[below] {$e_6$};
        \draw (4.75,1.25) node[left] {$e_7$};
        \end{tikzpicture}
        \caption{$T_0$}
    \end{subfigure}
    \begin{subfigure}[b]{0.3\linewidth}
        \begin{tikzpicture}[scale = .5, thick]
        
        \draw (0,0)--(2,2);
        \draw (2,2)--(4,2);
        \draw (4,2)--(6,0);
        \draw (2,2)--(1,3);
        \draw (4,2)--(5,3);
        
        \draw [fill] (0,0) circle [radius = .1];
        \draw [fill] (1,1) circle [radius = .1];
        \draw [fill] (2,2) circle [radius = .1];
        \draw [fill] (1,3) circle [radius = .1];
        \draw [fill] (4,2) circle [radius = .1];
        \draw [fill] (5,3) circle [radius = .1];
        \draw [fill] (5,1) circle [radius = .1];
        \draw [fill] (6,0) circle [radius = .1];
        
        \draw (0,0) node[below left] {$1$};
        \draw (1,3) node[above left] {$4$};
        \draw (5,3) node[above right] {$2$};
        \draw (6,0) node[below right] {$3$};
        
        \draw (.25,.25) node[right] {$e_1$};
        \draw (1.25,2.75) node[right] {$e_4$};
        \draw (4.75,2.75) node[left] {$e_2$};
        \draw (5.75,.25) node[left] {$e_3$};
        \draw (1.25,1.25) node[right] {$e_8$};
        \draw (3,2) node[below] {$e_7$};
        \draw (4.75,1.25) node[left] {$e_6$};
        \end{tikzpicture}
        \caption{$T_1$}
    \end{subfigure}
    \caption{A 4 leaf sunlet network $N$ and the two trees $T_0$ and $T_1$ that are obtained by deleting the reticulation edges $e_8$ and $e_5$ respectively.}
    \label{fig:SunletAndTrees}
\end{figure}
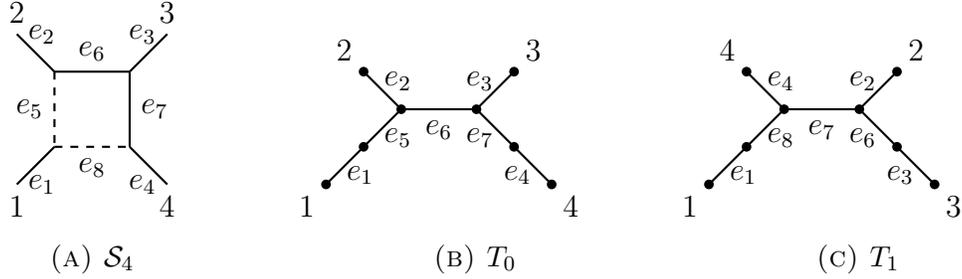

The two directed edges which point into the reticulation vertex are often called \emph{reticulation edges} and are drawn as dotted edges instead of directed edges since they are implicitly directed toward the vertex at which they meet. This is illustrated in Figure \ref{fig:SunletAndTrees}. Observe that deleting either of the reticulation edges from the sunlet network yields a tree. These underlying trees are used to construct the parameterization of the network model. For any phylogenetic model $\psi_T$ which is defined for trees, it is naturally extended to a sunlet network $N$ by defining
\[
\psi_N = \lambda \psi_{T_0} + (1-\lambda) \psi_{T_1}
\]

We now focus on the concrete problem of computing the ideal of phylogenetic invariants for a 4-leaf sunlet network $N_4$ under the K3P model. The K3P model is the generic group-based model for the group $G = \Zz_2 \times \Zz_2$. This means for each edge of the network $N_4$ and each $g \in \Zz_2 \times \Zz_2$ we have a parameter $a_g^e$. The parameterization $\psi_{N_4}$ is then given by
\[
q_{g_1, g_2, g_3, g_4} \mapsto \begin{cases}
a_{g_1}^1 a_{g_2}^2 a_{g_3}^3 a_{g_4}^4 a_{g_1}^5 a_{g_1 + g_2}^6 a_{g_4}^7 +
a_{g_1}^1 a_{g_2}^2 a_{g_3}^3 a_{g_4}^4 a_{g_3}^6 a_{g_1 + g_4}^7 a_{g_1}^8 
& \mbox{ if } \sum_{i \in [4]}g_i = 0 \\ 
0 & \mbox{ otherwise}
\end{cases}
\]
Since in this case $G = \Zz_2 \times \Zz_2$, there are a total of $|G|^4 = 4^{4-1} = 64$ variables in the domain of $\psi_{N_4}$ and $4 \times 8 = 32$ parameters; however, by exploiting the fact the associated map of varieties is actually of the form $\psi_T: \prod_{i = 1}^8 \Pp^3 \to \Pp^{63}$, one can naturally reparameterize with only $8 \times (4 - 1) + 1 = 25$ parameters. This means that in total the elimination ideal will be in $89$ variables. Recently, the authors of \cite{martin2023algebraic} attempted to find all generators up to total degree 3 in $\ker(\psi_{N_4})$ using standard Gr\"obner basis algorithms in Macaulay2 with degree-limiting. They were able to find all degree 2 generators however after 100 days the Gr\"obner basis algorithm still did not terminate to provide all degree 3 generators. 

We ran our Macaulay2 implementation of Algorithm \ref{alg:componentsOfKernel} which has no parallelization features on a MacBook Pro with an Apple M2 chip and 16 GB of RAM. It takes slightly over 8 minutes for Algorithm \ref{alg:componentsOfKernel} to produce all minimal generators of $\ker(\psi_{N_4})$ of total degree at most 3. We also ran this computation without the speed-up from \cref{cor:MatroidComponentSkip}. For the degree 2 generators the computation time was quite similar however for the degree 3 generators the computation took approximately 30 minutes instead of 8. The final results are summarized in the following theorem. 

\begin{theorem}
\label{thm:K3P_Gens4}
The ideal of phylogenetic invariants $I_{N_4} = \ker(\psi_{N_4})$ for K3P model on a four leaf sunlet network has $12$ minimal quadratic and $64$ minimal cubics generators. 
\end{theorem}

We were actually able to compute all minimal degree 2 generators for 5-leaf sunlets as well. In this case $\psi_{N_5}$ maps from a ring in $256$ variables into a ring with $31$ variables so the elimination ideal is in $287$ variables total. Despite this our algorithm is still able to compute all degree 2 generators in only 25 minutes and with parallelization could compute all degree 3 generators as well based on our current benchmarking. As one can see, Algorithm \ref{alg:componentsOfKernel} can scale to extremely large numbers of variables provided that the generators of interest are of low total degree and the map is homogeneous in a reasonably fine multigrading. The results are summarized in the following theorem and broken down in Table \ref{table:K3P_Gens} below.

\begin{theorem}
\label{thm:K3P_Gens5}
There are 648 minimal quadratic invariants of the K3P model on a 5-leaf sunlet network.  
\end{theorem}

\begin{table}
\centering
\begin{tabular}{|p{1.2cm}||p{2.3cm}|p{2.3cm}|p{1.7cm}|p{2.3cm}|p{2.3cm}|p{1cm}|p{1.5cm}|}
\hline
\multicolumn{8}{|c|}{Minimal Generators for 4 and 5 Leaf Sunlet Networks} \\ 
\hline
Leaves & Total Degree & Monomials & Grading Rank & Multidegrees & Skipped Components &  Min. Gens. & Time (sec) \\
\hline
\hline
4 & 2 & 2080 & 13 & 1720 & 1708 & 12 & 9.66 \\ 
4 & 3 & 45,760 & 13 & 25,152 & 24,304 & 64 & 492.31 \\
\hline
\hline
5 & 2 & 32,896 & 16 & 19,936 & 19,312 & 648 & 1504.03 \\ 
5 & 3 & 2,829,056 & 16 & 637,440 & - & - & -\\
\hline
\end{tabular}
\caption{This table shows the number of monomials, distinct multidegrees, the number of minimal generators, and the time each computation took in each total degree for both 4 and 5 leaf sunlet networks. The column skipped components corresponds to the number of components which can be skipped completely using \cref{cor:MatroidComponentSkip}. }
\label{table:K3P_Gens}
\end{table}

\subsection{The TN93 Model on a 4-Leaf Tree}
As discussed in the previous section, group-based models for trees have many nice algebraic properties associated to them. In particular, there is a linear change of coordinates which realizes the associated varieties as toric varieties. In practice, these models may not be the most biologically relevant. For example, it might not be a reasonable assumption for the root distribution $\pi$ to be uniform. 

In this section, we consider the Timura-Nei (TN93) model \cite{TN93} as studied in \cite{casanellas2023novel} and compute all of the quadratic invariants for a $4$-leaf tree. This model is \emph{algebraic time-reversible} meaning that for each transition matrix $M$, we have that 
\[
    \pi_i M_{i,j} = \pi_j M_{j,i}
\]
and that the collection of transition matrices all commute with each other. These assumptions amount to saying that the transition matrices are simultaneously diagonalizable and that the root distribution $\pi$ is an eigenvector of $M^T$ with eigenvalue 1. The TN93 model enjoys much more flexibility than group-based models.

\begin{definition}
    The TN93 model is a $4$-state algebraic time-reversible model with transition matrices of the form 
    \[
    \begin{pmatrix}
        *_1 & \pi_2 c & \pi_3 b & \pi_4 b \\
        \pi_1 c & *_2 & \pi_3 b & \pi_4 b \\
        \pi_1 b & \pi_2 b & *_3 & \pi_4 d \\
        \pi_1 b & \pi_2 b & \pi_3 d & *_4
    \end{pmatrix}
    \]
    where $*_i$ is chosen so that each row sums to 1 where the root distribution is $\pi = (\pi_1,\pi_2,\pi_3,\pi_4)$.
\end{definition}

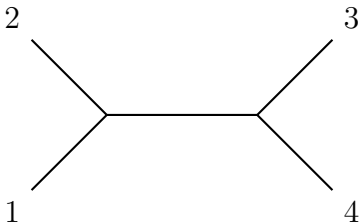
\begin{figure}
        \begin{tikzpicture}[scale = 1, thick]
        \draw (1,1)--(2,2);
        \draw (2,2)--(4,2);
        \draw (4,2)--(5,1);
        \draw (2,2)--(1,3);
        \draw (4,2)--(5,3);

        \draw (1,1) node[below left] {$1$};
        \draw (1,3) node[above left] {$2$};
        \draw (5,3) node[above right] {$3$};
        \draw (5,1) node[below right] {$4$};
        
        \end{tikzpicture}
        \caption{A 4-leaf tree with one non-trivial split $12|34$.}
        \label{figure: TN93 quartet}
\end{figure}

We will focus on the quartet tree $T$ which is pictured in Figure \ref{figure: TN93 quartet} under the TN93 model. Since the transition matrices are simultaneously diagonalizable, if we ignore the stochastic restrictions on these matrices, we see that the variety is parameterized by the $5 \times 4 = 20$ eigenvalues of these matrices. We also assume that the root distribution is fixed and generic, so instead of working over $\Cc$, we work over the fraction field $K = \Cc(\pi_1,\pi_2,\pi_3,\pi_4)$. These observations along with the fact that this is a $4$-state model means that the parametrization takes the following form.
\[
    \phi_T : K[p_{i_1,i_2,i_3,i_4} ~|~ (i_1,i_2,i_3,i_4) \in [4]^4] \to K[\lambda_{1,1}, \dotsc, \lambda_{5,4}]
\]

In \cite{casanellas2023novel}, the authors describe a linear change of coordinates from the probability coordinates to a new set of coordinates $q_{i_1,i_2,i_3,i_4}$ which has two nice properties: (1) 176 of the $q_{i_1,i_2,i_3,i_4}$'s map to 0 and (2) 71 of the remaining non-zero coordinates are monomials in the eigenvalues of the transition matrices. We will refer to the set of indices of the 80 non-zero coordinates by $\mathcal{NZ}_T$. In particular, we can greatly reduce the number of variables in the elimination ideal from 276 to just 100. The new parametrization takes the following form.
\[
    \varphi_T : K[q_{i_1,i_2,i_3,i_4} ~|~ (i_1,i_2,i_3,i_4) \in \mathcal{NZ}_T] \to K[\lambda_{1,1}, \dotsc, \lambda_{5,4}] 
\]
We let $I_T$ denote the kernel of $\varphi_T$. The authors showed that on an open set of $\mathcal{V}(I_T)$ the variety is a complete intersection and is cut out by 64 equations of degree at most 4 \cite[Theorem 5.14]{casanellas2023novel}.

While the number of parameters is greatly reduced from the general Markov model, computing a Gr{\"o}bner basis for $I_T$ is probably still out of reach. However, using Algorithm \ref{alg:componentsOfKernel}, we found all minimal quadrics in $I_T$. We see that there are many more minimal quadrics cutting out the full variety.
 
\begin{theorem}
    There are 375 minimal quadratic invariants of $T$ under the TN93 model. 
\end{theorem}

\section*{Acknowledgements}
Part of this research was performed while the authors were visiting the Institute for Mathematical and Statistical Innovation (IMSI), which is supported by the National Science Foundation (Grant No. DMS-1929348). Benjamin Hollering was supported by the Alexander von Humboldt Foundation. Joseph Cummings was supported by NSF CCF-1812746.

\bibliographystyle{plain}
\bibliography{refs.bib}

\begin{thebibliography}{10}

\bibitem{allman2010identifiability}
Elizabeth~S Allman, Sonia Petrovic, John~A Rhodes, and Seth Sullivant.
\newblock Identifiability of two-tree mixtures for group-based models.
\newblock {\em IEEE/ACM transactions on computational biology and bioinformatics}, 8(3):710--722, 2010.

\bibitem{allman2003invariants}
Elizabeth~S. Allman and John~A. Rhodes.
\newblock Phylogenetic invariants for the general markov model of sequence mutation.
\newblock {\em Mathematical Biosciences}, 186(2):113--144, 2003.

\bibitem{allman2008gmm}
Elizabeth~S. Allman and John~A. Rhodes.
\newblock Phylogenetic ideals and varieties for the general markov model.
\newblock {\em Advances in Applied Mathematics}, 40(2):127--148, 2008.

\bibitem{bates2011salmon}
Daniel~J. Bates and Luke Oeding.
\newblock Toward a salmon conjecture.
\newblock {\em Exp. Math.}, 20(3):358--370, 2011.

\bibitem{msolve}
J\'{e}r\'{e}my Berthomieu, Christian Eder, and Mohab Safey El~Din.
\newblock Msolve: A library for solving polynomial systems.
\newblock ISSAC '21, page 51–58, New York, NY, USA, 2021. Association for Computing Machinery.

\bibitem{casanellas2023novel}
Marta Casanellas, Roser~Homs Pons, and Angélica Torres.
\newblock A novel algebraic approach to time-reversible evolutionary models, 2023.

\bibitem{chifman2014quartet}
Julia Chifman and Laura Kubatko.
\newblock {Quartet Inference from SNP Data Under the Coalescent Model}.
\newblock {\em Bioinformatics}, 30(23):3317--3324, 08 2014.

\bibitem{coxlittleoshea}
David~A. Cox, John Little, and Donal O'Shea.
\newblock {\em Ideals, varieties, and algorithms}.
\newblock Undergraduate Texts in Mathematics. Springer, Cham, fourth edition, 2015.
\newblock An introduction to computational algebraic geometry and commutative algebra.

\bibitem{cummings2023multigraded}
Joseph Cummings and Jonathan Hauenstein.
\newblock Multi-graded macaulay dual spaces, 2023.

\bibitem{cummings2021invariants}
Joseph Cummings, Benjamin Hollering, and Christopher Manon.
\newblock Invariants for level-1 phylogenetic networks under the cavendar-farris-neyman model, 2021.

\bibitem{draisma2009equivariant}
Jan Draisma and Jochen Kuttler.
\newblock On the ideals of equivariant tree models.
\newblock {\em Math. Ann.}, 344(3):619--644, 2009.

\bibitem{eriksson2005tree}
Nicholas Eriksson.
\newblock Tree construction using singular value decomposition.
\newblock In {\em Algebraic statistics for computational biology}, pages 347--358. Cambridge Univ. Press, New York, 2005.

\bibitem{evans1993invariants}
Steven~N. Evans and T.~P. Speed.
\newblock Invariants of some probability models used in phylogenetic inference.
\newblock {\em Ann. Statist.}, 21(1):355--377, 1993.

\bibitem{faugere2002f5}
Jean~Charles Faug\`{e}re.
\newblock A new efficient algorithm for computing gr\"{o}bner bases without reduction to zero (f5).
\newblock In {\em Proceedings of the 2002 International Symposium on Symbolic and Algebraic Computation}, ISSAC '02, page 75–83, New York, NY, USA, 2002. Association for Computing Machinery.

\bibitem{faugere2016complexity}
Jean-Charles Faugère, Mohab {Safey El Din}, and Thibaut Verron.
\newblock On the complexity of computing gröbner bases for weighted homogeneous systems.
\newblock {\em Journal of Symbolic Computation}, 76:107--141, 2016.

\bibitem{friedlan2012salmon}
Shmuel Friedland and Elizabeth Gross.
\newblock A proof of the set-theoretic version of the salmon conjecture.
\newblock {\em J. Algebra}, 356:374--379, 2012.

\bibitem{M2}
Daniel~R. Grayson and Michael~E. Stillman.
\newblock Macaulay2, {V}ersion 1.20, 2022.
\newblock {\tt http://www.math.uiuc.edu/Macaulay2/}.

\bibitem{gross2018distinguishing}
Elizabeth Gross and Colby Long.
\newblock Distinguishing phylogenetic networks.
\newblock {\em SIAM Journal on Applied Algebra and Geometry}, 2(1):72--93, 2018.

\bibitem{hendy1996complete}
Michael~D Hendy and David Penny.
\newblock Complete families of linear invariants for some stochastic models of sequence evolution, with and without the molecular clock assumption.
\newblock {\em Journal of Computational Biology}, 3(1):19--31, 1996.

\bibitem{hollering2021identifiability}
Benjamin Hollering and Seth Sullivant.
\newblock Identifiability in phylogenetics using algebraic matroids.
\newblock {\em J. Symbolic Comput.}, 104:142--158, 2021.

\bibitem{jensen2008ComputingGF}
Anders~Nedergaard Jensen.
\newblock Computing gr{\"o}bner fans and tropical varieties in gfan.
\newblock 2008.

\bibitem{TN93}
Tamura K and Nei M.
\newblock Estimation of the number of nucleotide substitutions in the control region of mitochondrial dna in humans and chimpanzees.
\newblock {\em Mol Biol Evol.}, 10(3):512--26, 1993 May.

\bibitem{CCA2}
Martin Kreuzer and Lorenzo Robbiano.
\newblock {\em Computational commutative algebra. 2}.
\newblock Springer-Verlag, Berlin, 2005.

\bibitem{long2015identifiability}
Colby Long and Seth Sullivant.
\newblock Identifiability of 3-class {J}ukes-{C}antor mixtures.
\newblock {\em Adv. in Appl. Math.}, 64:89--110, 2015.

\bibitem{martin2023algebraic}
Samuel Martin, Vincent Moulton, and Richard~M. Leggett.
\newblock Algebraic invariants for inferring 4-leaf semi-directed phylogenetic networks.
\newblock {\em bioRxiv}, 2023.

\bibitem{mateusz2011geometry}
Mateusz Micha\l~ek.
\newblock Geometry of phylogenetic group-based models.
\newblock {\em J. Algebra}, 339:339--356, 2011.

\bibitem{mathrepo}
{MATHREPO} {Mathematical Data and Software}.
\newblock \url{https://mathrepo.mis.mpg.de/MultigradedImplicitization}, 2023.
\newblock [Online; accessed 1 November 2023].

\bibitem{OSCAR}
Oscar -- open source computer algebra research system, version 0.14.0-dev, 2023.

\bibitem{Pachter-Sturmfels04}
Lior Pachter and Bernd Sturmfels.
\newblock Tropical geometry of statistical models.
\newblock {\em Proc. Natl. Acad. Sci. USA}, 101(46):16132--16137, 2004.

\bibitem{rosen2014computing}
Zvi Rosen.
\newblock Computing algebraic matroids.
\newblock {\em arXiv preprint arXiv:1403.8148}, 2014.

\bibitem{steel2016phylogeny}
Mike Steel.
\newblock {\em Phylogeny: discrete and random processes in evolution}.
\newblock SIAM, 2016.

\bibitem{sturma2023testing}
Nils Sturma, Mathias Drton, and Dennis Leung.
\newblock Testing many constraints in possibly irregular models using incomplete u-statistics, 2023.

\bibitem{sturmfels2005toric}
Bernd Sturmfels and Seth Sullivant.
\newblock Toric ideals of phylogenetic invariants.
\newblock {\em Journal of Computational Biology}, 12(2):204--228, 2005.

\bibitem{algstat}
Seth Sullivant.
\newblock {\em Algebraic statistics}, volume 194 of {\em Graduate Studies in Mathematics}.
\newblock American Mathematical Society, Providence, RI, 2018.

\end{thebibliography}
\end{document}